\documentclass[12pt]{article}
\usepackage{amsmath}
\usepackage[affil-it]{authblk}
\usepackage{amsfonts}
\usepackage{amsthm}
\usepackage{amssymb}
\usepackage{makecell}
\usepackage{appendix}
\usepackage{algorithmic}
\usepackage{graphicx}
\usepackage{algorithm2e}
\usepackage{tabu}
\usepackage{color}
\definecolor{darkgreen}{rgb}{0.0, 0.63, 0.0}
\theoremstyle{definition}
\newtheorem{definition}{Definition}[section]
\usepackage{mathtools}

\theoremstyle{lemma}
\newtheorem{theorem}{Theorem}
\newtheorem{corollary}{Corollary}[theorem]
\newtheorem{lemma}[theorem]{Lemma}

\DeclareMathOperator*{\argmin}{arg\,min}

\newcommand{\vol}{\text{Vol}}
\newcommand{\norm}{\text{Nm}}

\newcommand{\trace}{\text{Tr}}

\usepackage[margin=0.9in]{geometry}

\begin{document}
\title{Unit Reducible Fields and Perfect Unary Forms}
\author{Alar Leibak, Christian Porter, Cong Ling}
\maketitle
\begin{abstract}
    In this paper, we introduce the notion of unit reducibility for number fields, that is, number fields in which all positive unary forms attain their nonzero minimum at a unit. Furthermore, we investigate the link between unit reducibility and the number of homothety classes of perfect unary forms for a given number field, and prove an open conjecture about the number of classes of perfect unary forms in real quadratic fields, stated by D. Yasaki.
\end{abstract}
\section{Introduction}
Let $K$ be a real algebraic number field of degree $n$ over $\mathbb{Q}$ with ring of integers $\mathcal{O}_K$. Associate to $K$ the canonical embeddings $\sigma_1,\dots,\sigma_n$ into $\mathbb{R}$. A quadratic form $f: K^m \to K$, defined by

\[
f(x_1,\ldots,x_m) = \sum_{k,l=1}^{m} f_{kl}x_kx_l, \quad f_{kl}=f_{lk}\in K,
\]
of rank $m$ is said to be positive definite if

\[
\sigma_i(f)(x_1,\ldots,x_m) = \sum_{k,l} \sigma_i(f_{kl})x_kx_l
\]
is positive definite for each $i =1,\dots,n$. Throughout the paper, we will simply refer to positive definite quadratic forms as quadratic forms, unless we need to clarify.

If $f$ is defined by a single variable in $K$, the corresponding quadratic form is called a \emph{unary quadratic form}. If $f$ generates a positive-definite quadratic form, we say that the element $f$ is \emph{totally positive}. We denote by $K_{>>0}$ the set of totally positive elements of $K$, i.e $K_{>>0}=\{a\in K \ | \ \sigma_i(a)>0, \ i=1, \ldots,n \}$. Let $\omega_1=1, \omega_2,\ldots, \omega_n$ be an integral basis of $\mathcal{O}_K$. Note that for all $a \in K_{>>0}$,
\begin{align*}
    q_a: K \to \mathbb{Q}, \hspace{2mm} &q_a(x_1, x_2, \ldots, x_n)=\trace_{K/\mathbb{Q}}(a(x_1\omega_1+ x_2\omega_2 +\ldots+ x_n\omega_n)^2) \\&= \sum_{k,l=1}^{n} \trace_{K/\mathbb{Q}}(a\omega_k\omega_l)x_kx_l
\end{align*}
corresponds to an $n$-dimensional  quadratic form with rational coefficients, known as the trace quadratic form. For a totally positive $a\in K$, we will denote by \[\mu(a)=\min_{x \in \mathcal{O}_K \setminus \{0\}}  \trace_{K/\mathbb{Q}} (ax^2) = \min_{(x_1,x_2,\ldots, x_n)\in \mathbb{Z}^n\setminus \{0\}} q_a(x_1, x_2,\ldots, x_n), \] 
and $\mathcal{M}(a)=\{x \in \mathcal{O}_K: \trace_{K/\mathbb{Q}}(ax^2)=\mu(a)\}$.

Let $\mathcal{O}_K^\times$ denote the unit group of $\mathcal{O}_K$. We say that a unary form defined by $a \in K_{>>0}$ is \emph{reduced} if $\trace_{K/\mathbb{Q}}(a) \leq \trace_{K/\mathbb{Q}}(au^2)$, for all $u \in \mathcal{O}_K^\times$. We denote the set of all $a \in K_{>>0}$ such that $a$ corresponds to a reduced unary form by $\mathcal{F}_K$, and call this set the \emph{reduction domain} of unary forms in $K$. Note that every unary form is equivalent to a reduced unary form.

Finally, we give the following important definition.
\begin{definition}
Let $K$ be a totally real number field with ring of integers $\mathcal{O}_K$ and unit group $\mathcal{O}_K^\times$. We say that $K$ is \emph{unit reducible} if, for all $a \in K_{>>0}$,
\begin{align*}
    \mu(a)=\min_{u \in \mathcal{O}_K^\times} \trace_{K/\mathbb{Q}}(au^2).
\end{align*}
\end{definition}
Unit reducibility is a useful property in the study of lattice-based cryptography \cite{por21}. The security of many lattice-based cryptosystems is underpinned by the so-called ``shortest vector problem'' (SVP), which asks the adversary to find a shortest nonzero vector of a lattice given an arbitrary lattice basis. Lattices defined over number fields are increasingly used in lattice-based cryptography due to their efficiency. Lattice reduction over number fields is a common strategy to find short vectors in such lattices \cite{euclidean,lllmodules}. Some  lattices are constructed from principal ideals of number fields \cite{cramer}, and so the problem translates to finding the shortest nonzero arithmetical value of a unary quadratic form. When the field is unit reducible, this suggests that finding the smallest arithmetical nonzero value of a unary quadratic form corresponds also to finding the shortest generator of the principal ideal. Although at high dimensions the shortest generator is unlikely to be a shortest vector of a principal ideal lattice, this can be true at small dimensions. Finding unit reducible number fields can lead to algorithms of lattice reduction at high dimensions, the latter being an ongoing subject of study.

A unary form $ax^2$ is said to be \emph{perfect} if it is uniquely determined by $\mu(a),\mathcal{M}(a)$. It is immediately clear that if $ax^2$ is perfect, then for any $\lambda>0 \in \mathbb{Q}$, $\mu(\lambda a)=\lambda \mu(a)$, $\mathcal{M}(\lambda a)= \mathcal{M}(a)$, and so perfect unary forms can be considered by their homothety classes. We denote by $n_K$ the number of $\text{GL}_1(\mathcal{O})$-inequivalent homothety classes of perfect forms of a real number field $K$. In this paper, we establish a method of determining the number of classes of perfect forms by studying unit reducible fields. The main result of the paper is as follows.
\begin{theorem}\label{thm1}
Let $K=\mathbb{Q}(\sqrt{d})$ be a real quadratic field for some positive, square-free integer $d$ not equal to $1$. Then $n_K=1$ if and only if $d$ is of one of the following types:
\begin{itemize}
    \item $T1$: $d=m^2+1$, $m \in \mathbb{N}$, $m$ odd,
    \item $T2$: $d=m^2-1$, $m \in \mathbb{N}$, $m$ even,
    \item $T3$: $d=m^2+4$, $m \in \mathbb{N}$, $m$ odd,
    \item $T4$: $d=m^2-4$, $m \in \mathbb{N}$, $m>3$ odd.
\end{itemize}
\end{theorem}

\noindent\textbf{Remark 1.} Watanabe et. al. asked are there infinitely many real quadratic fields with only one equivalence class of unary perfect forms (see \cite{W})? Yasaki proved that all real quadratic fields $ \mathbb{Q}(\sqrt{d}) $, $ d=(2k+1)^2+1 $ is square-free, has that property \cite[Corollary 5.1]{Y}. Moreover, his computations showed that for square-free positive  $ d<200000 $ there were four disjoint families of real quadratic fields with one class of unary perfect forms \cite[Remark at p. 773]{Y}: $T1$, $T2$, $T3$
and $T4$. Our main result (Theorem \ref{thm1}) proves Yasaki's remark.

A real quadratic field $\mathbb{Q}(\sqrt{d})$ ($d>0$ is square-free) is said to be of \emph{Richaud-Degert type} (R-D type) if $d=m^2+r$, where $m,r$ are integers, $m$ positive, $-m < r \leq m$ and $4m \equiv 0 \mod r$ \cite{R-D}. Using similar methods, we also ascertain the following results.
\begin{theorem}\label{thm2}
Let $K=\mathbb{Q}(\sqrt{d})$ be a real quadratic field of R-D type and let $d \equiv 2,3 \mod 4$, or $d \equiv 1 \mod 4$ and $n$ odd. If $K$ is not unit reducible, then $n_K=2$.
\end{theorem}
\begin{theorem}\label{thm3}
Let $K$ be the totally real cubic number field with defining polynomial $f(x)=x^3-tx^2-(t+3)x-1$, where $t \geq 0$. Then $n_K=2$ if the ring of integers of $K$ is monogenic.
\end{theorem}
\subsection{Paper organisation}
Section $2$ is split into two parts. In the first, we determine all the classes of unit reducible fields for real quadratic fields using arguments from classical reduction theory. In the second part, we prove that all simplest cubic fields where the ring of integers is monogenic is unit reducible by determining the so-called reduction domain of $K$. In section $3$, using Minkowski's first theorem, we determine a condition which, if satisfied, confirms that a totally general, real number field is not unit reducible. Finally, in section $4$, we conclude by proving theorems $1$ to $3$.
\section{Families of Unit Reducible Fields}\label{sec2}
\subsection{Real quadratic fields}
Let $K=\mathbb{Q}(\sqrt{d})$ for some positive square-free integer $d \neq 1$. The discriminant $\Delta_K$ of the field $K$ is
\begin{align*}
    \Delta_K=\begin{cases}
    &4d \hspace{2mm} \text{if} \hspace{2mm} d \equiv 2,3 \mod 4,
    \\
    &d \hspace{2mm} \text{if} \hspace{2mm} d \equiv 1 \mod 4
    \end{cases},
\end{align*}
and the ring of integers $\mathcal{O}_K$ has the representation $\mathcal{O}_K=\mathbb{Z}[\omega]$, where
\begin{align*}
    \omega=\begin{cases}
    &\sqrt{d} \hspace{2mm} \text{if} \hspace{2mm}d \equiv 2,3 \mod 4,
    \\& \frac{1+\sqrt{d}}{2} \hspace{2mm} \text{if} \hspace{2mm} d \equiv 1 \mod 4
    \end{cases}.
\end{align*}
We associate two embeddings to $\mathbb{Q}(\sqrt{d})$, the trivial embedding $\sigma_1$, and
\begin{align*}
    \sigma_2(x_1+x_2\sqrt{d})=x_1-x_2\sqrt{d},
\end{align*}
where $x_1,x_2 \in \mathbb{Q}$, for any $x_1+x_2\sqrt{d} \in K$.
\begin{lemma}[\cite{lei02}, Lemma 10]\label{lm1}
Let $K$ be a real quadratic field and denote by $u$ the fundamental unit of $K$ satisfying $u>1$. Let $a$ denote a totally positive element of $K$. Then $ax^2$ is reduced if and only if
\begin{align*}
    &\emph{\trace}_{K/\mathbb{Q}}(a) \leq \emph{\trace}_{K/\mathbb{Q}}(au^2),
    \\&\emph{\trace}_{K/\mathbb{Q}}(a) \leq \emph{\trace}_{K/\mathbb{Q}}(au^{-2}).
\end{align*}
\end{lemma}
Using the notation in Lemma \ref{lm1}, set $a=a_1+a_2\sqrt{d} \in K_{>>0}$, $a_1,a_2 \in \mathbb{Q}$ and let $u^2=u_1+u_2\sqrt{d}$ where $u_1,u_2>0$, so $u^{-2}=u_1-u_2\sqrt{d}$. Then the form $ax^2$ is reduced if and only if
\begin{align*}
    2a_1 \leq 2u_1a_1 - 2du_2|a_2|,
\end{align*}
and so rearranging gives
\begin{align*}
    \frac{|a_2|}{a_1} \leq \frac{u_1-1}{du_2}.
\end{align*}
We are now equipped to prove the following theorem.
\begin{theorem}
A real quadratic field $K=\mathbb{Q}(\sqrt{d})$ is unit reducible if and only if $d$ is one of the four types listed in theorem \ref{thm1}.
\end{theorem}
\begin{proof}
Let $a=a_1+a_2\sqrt{d}$ denote a totally positive element. Without loss of generality, we may assume that $a$ is reduced and that $a_2 \geq 0$, so
\begin{align}
    \frac{a_2}{a_1} \leq \frac{u_1-1}{du_2}, \label{ineq1}
\end{align}
where $u^2=u_1+u_2\sqrt{d}$, $u$ is the fundamental unit of $K$ satisfying $u>1$. The binary rational quadratic form generated by $q_a$ is given by
\begin{align*}
    \trace_{K/\mathbb{Q}}(a(x_1+x_2\sqrt{d})^2)=2(x_1^2+dx_2^2)a_1+4dx_1x_2a_2.
\end{align*}
Assume first that $K$ is of the type $T_1$ or $T_2$, i.e. $d \equiv 2,3 \mod 4$. Then we have
\begin{align*}
    \min_{x \in \mathcal{O}_K \setminus \{0\}} \trace_{K/\mathbb{Q}}(ax^2)= \min_{(0,0) \neq (x_1,x_2) \in \mathbb{Z}^2}\left\{2(x_1^2+dx_2^2)a_1+4dx_1x_2a_2\right\} \triangleq 2g(x_1,x_2).
\end{align*}
Now, set $h(x_1,x_2)=g(x_1-kx_2,x_2)$ where $k$ is the nearest rational integer to $da_2/a_1$. Then
\begin{align*}
    h(x_1,x_2)=a_1x_1^2+2(da_2-ka_1)x_1x_2+((k^2+d)a_1-2dka_2)x_2^2.
\end{align*}
Recall that a binary rational quadratic form $f(x_1,x_2)=f_{11}x_1^2+f_{12}x_1x_2+f_{22}x_2^2$ is reduced (in the sense of rational forms) if and only if $|f_{12}| \leq \min\{f_{11},f_{22}\}$. Since $2|da_2-ka_1| \leq a_1$ by the definition of $k$, the form $h$ is reduced if and only if
\begin{align}
    2|da_2-ka_1| \leq (k^2+d)a_1-2dka_2. \label{ineq2}
\end{align}
In turn, proving that $h$ is a reduced binary rational quadratic form for any values of $a_1,a_2$ (assuming that $a_1+\sqrt{d}a_2 \in \mathcal{F}_K$) will verify that $K$ is unit reducible, since the minimum of a reduced binary rational quadratic form $f(x_1,x_2)$ is attained at $f(1,0)$, and since $h(1,0)=g(1,0)=\frac{1}{2}\trace_{K/\mathbb{Q}}(a)$, $\trace_{K/\mathbb{Q}}(a)=\mu(a)$ for any reduced unary form $ax^2$.
\\
Suppose first that $d$ is of the type $T_1$, so $d= m^2+1$ where $m$ is an odd positive integer. Then the fundamental unit $u$ of $K$ satisfying $u>1$ is given by $u=m+\sqrt{d}$, so according to inequality \eqref{ineq1},
\begin{align*}
    \frac{a_2}{a_1} \leq \frac{(2m^2+1)-1}{2m(m^2+1)}=\frac{m}{m^2+1},
\end{align*}
and so
\begin{align*}
    &(k^2+d)a_1-2dka_2-2(da_2-ka_1)=(k^2+2k+m^2+1)a_1-2(m^2+1)(k+1)a_2\\&\geq \left(k^2+2k+m^2-2km\right)a_1=((k-m)^2+2k)a_1 \geq 0,
\end{align*}
since $k$ must necessarily be greater than or equal to zero by the assumption that $a_2 \geq 0$, which verifies that inequality \eqref{ineq2} holds, and so $h$ is a reduced rational binary quadratic form when $d$ is of the type $T_1$.
\\
Suppose that $d$ is of the type $T_2$, so $d=m^2-1$ for some positive even integer $m$. Then the fundamental unit $u$ of $K$ satisfying $u>1$ is given by $u=m+\sqrt{d}$, so according to inequality \eqref{ineq1},
\begin{align*}
    \frac{a_2}{a_1} \leq \frac{(2m^2-1)-1}{2m(m^2-1)}=\frac{1}{m},
\end{align*}
and so
\begin{align*}
    &(k^2+d)a_1-2dka_2-2(da_2-ka_1)=(k^2+m^2+2k-1)a_1-2(m^2-1)(k+1)a_2
    \\& \geq \left(k^2+m^2+2k-1-2(m^2-1)\frac{k+1}{m}\right)a_1=\left((k-m+1)^2+\frac{2}{m}(k-m+1)\right)a_1 \geq 0.
\end{align*}
This verifies that inequality \eqref{ineq2} holds, and so $h$ is a reduced rational binary quadratic form when $d$ is of the type $T_2$.
\\
When $d$ is of type $T_3,T_4$, we have $d \equiv 1 \mod 4$, and so $\mathcal{O}_K=\mathbb{Z}\left[\frac{1+\sqrt{d}}{2}\right]$. Then
\begin{align*}
    \min_{x \in \mathcal{O}_K} \trace_{K/\mathbb{Q}}(ax^2)=2\min_{(0,0) \neq (x_1,x_2) \in \mathbb{Z}^2}\left\{a_1x_1^2+(a_1+da_2)x_1x_2+\left(\frac{1+d}{4}a_1+\frac{d}{2}a_2\right)x_2^2\right\} \triangleq 2g(x_1,x_2)
\end{align*}
Then, setting $h(x_1,x_2)=g(x_1-kx_2,x_2)$ where $k$ is the nearest integer to $(a_1+da_2)/2a_1$,
\begin{align*}
    h(x_1,x_2)=a_1x_1^2+(a_1+da_2-2ka_1)x_1x_2+\left(\left(\frac{1+d}{4}+k^2\right)a_1+\frac{d}{2}a_2-k(a_1+da_2)\right)x_2^2.
\end{align*}
Following a similar argument to before, $h$ is a reduced rational binary quadratic form if and only if
\begin{align}
    a_1+da_2-2ka_1 \leq \left(\frac{1+d}{4}+k^2\right)a_1+\frac{d}{2}a_2-k(a_1+da_2), \label{ineq3}
\end{align}
and showing that $h$ is a reduced binary quadratic form for any $a \in \mathcal{F}_K$ corresponds to proving the field is unit reducible.
\\
Suppose that $d$ is of the type $T_3$, so $d=m^2+4$ for some positive odd integer $m$. Then if $m=2p+1$ for some $p \geq 0$, the fundamental unit $u$ of $K$ must be $u=p+\frac{1+\sqrt{d}}{2}$, so according to inequality \eqref{ineq1},
\begin{align*}
    \frac{a_2}{a_1} \leq \frac{\left(2p^2+2p+\frac{3}{2}\right)-1}{(4p^2+4p+5)\left(\frac{1}{2}+p\right)}=\frac{2p+1}{4p^2+4p+5},
\end{align*}
and so
\begin{align*}
    &\left(\frac{1+d}{4}+k^2\right)a_1+\frac{d}{2}a_2-k(a_1+da_2)-(a_1+da_2-2ka_1)\\&=\left(\frac{1+d}{4}+k^2+k-1\right)a_1-\left(-\frac{d}{2}+kd+d\right)a_2
    \\& \geq \left(\frac{4p^2+4p+6}{4}+k^2+k-1-\frac{4p^2+4p+5}{2}\frac{2p+1}{4p^2+4p+5}(2k+1)\right)a_1
    \\&=\left(\frac{4p^2+4p+6}{4}+k^2+k-1-\frac{2p+1}{2}(2k+1)\right)a_1=(p-k)^2a_1 \geq 0.
\end{align*}
This verifies that inequality \eqref{ineq3} holds, and so $h$ is a reduced binary quadratic form when $d$ is of the type $T_3$.
\\
Suppose that $d$ is of the type $T_4$, so $d=m^2-4$ for some positive odd integer, $m>3$. Then if $m=2p+1$ for some $p \geq 2$, the fundamental unit of $K$ must be $u=p+\frac{1+\sqrt{d}}{2}$, and so according to inequality \eqref{ineq1},
\begin{align*}
    \frac{a_2}{a_1} \leq \frac{\frac{1}{2}(4p^2+4p-1)-1}{(4p^2+4p-3)\frac{1}{2}(2p+1)}=\frac{1}{2p+1},
\end{align*}
and so
\begin{align*}
    &\left(\frac{1+d}{4}+k^2\right)a_1+\frac{d}{2}a_2-k(a_1+da_2)-(a_1+da_2-2ka_1)
    \\&=\left(\frac{1+d}{4}+k^2+k-1\right)a_1-\left(-\frac{d}{2}+kd+d\right)a_2
    \\& \geq \left(\frac{4p^2+4p-2}{4}+k^2+k-1-\frac{4p^2+4p-3}{2}\frac{1}{2p+1}\left(2k+1\right)\right)a_1
    \\&=\left((k-p)^2+\frac{4}{2p+1}(k-p)\right)a_1 \geq 0,
\end{align*}
since $p>1$. This verifies that inequality \eqref{ineq3} holds, and so $h$ is a reduced binary quadratic form when $d$ is of type $T_4$.
\\
It finally remains to show that $K$ is not unit reducible if $d$ is none of these types. Denote by $R_K$ the regulator of the field $K$. We claim that if 
\begin{align}
    \frac{\sqrt{\Delta_K}}{2\cosh(R_K)}<\frac{\pi}{4}, \label{ineq4}
\end{align}
then the field is not unit reducible. In order to show this claim holds, we will make use of the following lemma.
\begin{lemma}
Let $a=a_1+a_2\sqrt{d} \in K_{>>0}$ be an arbitrary totally positive element. Then $a \in \mathcal{F}_K$ if and only if the inequality
\begin{align}
    \frac{|a_2|}{a_1} \leq \frac{\tanh(R_K)}{\sqrt{d}} \label{ineq5}
\end{align}
holds.
\end{lemma}
\begin{proof}
By Lemma \ref{lm1} and the definition of the regulator of a real quadratic number field, $a$ is reduced if and only if
\begin{align*}
    &(a_1+a_2\sqrt{d})e^{2R_K}+(a_1-a_2\sqrt{d})e^{-2R_K} \geq a_1+a_2\sqrt{d}+a_1-a_2\sqrt{d}=2a_1,
    \\&(a_1+a_2\sqrt{d})e^{-2R_K}+(a_1-a_2\sqrt{d})e^{2R_K} \geq a_1+a_2\sqrt{d}+a_1-a_2\sqrt{d}=2a_1.
\end{align*}
These inequalities imply
\begin{align*}
    &a_1(\cosh(2R_K)-1)+a_2\sqrt{d}\sinh(2R_K) \geq 0,
    \\&a_1(\cosh(2R_K)-1)-a_2\sqrt{d}\sinh(2R_K) \geq 0,
\end{align*}
and so
\begin{align*}
    a_1(\cosh(2R_K)-1) \geq |a_2|\sqrt{d}\sinh(2R_K).
\end{align*}
The inequality in the statement of the lemma immediately follows.
\end{proof}
Suppose first that $K=\mathbb{Q}(\sqrt{d})$ for $d \equiv 2,3 \mod 4$. Then the associated rational quadratic form $g$ to the unary form generated by $a=a_1+a_2\sqrt{d}$ can be given by
\begin{align*}
    g(x_1,x_2) \triangleq \frac{1}{2}\trace_{K/\mathbb{Q}}(ax^2)=a_1(x_1^2+dx_2^2)+2da_2x_1x_2.
\end{align*}
If $K$ were not unit reducible, then there would exist a totally positive element $a$ satisfying
\begin{align*}
    \trace_{K/\mathbb{Q}}(ax^2)<\trace_{K/\mathbb{Q}}(a),
\end{align*}
for some nonzero $x \in \mathcal{O}_K$. In terms of the rational form $g$, this means that
\begin{align*}
    a_1(x_1^2+dx_2^2)+2da_2x_1x_2<a_1,
\end{align*}
for some nonzero integer lattice point $(x_1,x_2) \in \mathbb{Z}^2$. This can be rewritten as
\begin{align}
    x_1^2+dx_2^2+2d\frac{a_2}{a_1}x_1x_2<1.\label{ineq6}
\end{align}
Define the convex body $\mathcal{S}$ by
\begin{align*}
    \mathcal{S}=\left\{(y_1,y_2) \in \mathbb{R}^2: y_1^2+dy_2^2+2\frac{a_2d}{a_1}y_1y_2<1\right\}.
\end{align*}
It is clear that the set $\mathcal{S}$ is bounded, convex and symmetric about the origin. By Minkowski's convex body theorem, $\mathcal{S}$ contains a non-trivial integer solution if
\begin{align}
    \vol(\mathcal{S})>2^2 \text{disc}(\mathbb{Z}^2), \label{ineq7}
\end{align}
where $\text{disc}(\mathbb{Z}^2)$ denotes the volume or discriminant of the integer lattice $\mathbb{Z}^2$. Since
\begin{align*}
    \vol(\mathcal{S})=\frac{\pi}{\sqrt{d\left(1-\frac{a_2^2d}{a_1^2}\right)}},
\end{align*}
and $\text{disc}(\mathbb{Z}^2)=1$, we can write \eqref{ineq7} as
\begin{align*}
    \frac{\pi}{\sqrt{d\left(1-\frac{a_2^2d}{a_1^2}\right)}}>4.
\end{align*}
This yields
\begin{align*}
    \frac{\pi}{4}>\sqrt{d\left(1-\frac{a_2^2d}{a_1^2}\right)}.
\end{align*}
If $a \in \mathcal{F}_K$, then
\begin{align*}
    \frac{a_2^2}{a_1^2}\leq \frac{\tanh(R_K)^2}{d},
\end{align*}
and
\begin{align*}
    \sqrt{d\left(1-\frac{a_2^2d}{a_1^2}\right)} \geq \sqrt{d}\sqrt{1-\tanh(R_K)^2}=\frac{\sqrt{d}}{\cosh(R_K)}=\frac{\sqrt{\Delta_K}}{2\cosh(R_K)}.
\end{align*}
Now, suppose that $d \equiv 1 \mod 4$. Let $a=a_1+a_2\sqrt{d} \in K_{>>0}$ denote a totally positive element of $K=\mathbb{Q}(\sqrt{d})$ as before. Then the associated rational quadratic form $g$ to the unary form generated by $a$ can be given by
\begin{align*}
    g(x_1,x_2)\triangleq \frac{1}{2}\trace_{K/\mathbb{Q}}(ax^2)=a_1\left[\left(x_1+\frac{a_1+a_2d}{2a_1}x_2\right)^2+\frac{d}{4}\left(1-\frac{a_2^2d}{a_1^2}\right)x_2^2\right].
\end{align*}
Once again, we define the convex body
\begin{align*}
    \mathcal{S}=\{(y_1,y_2) \in \mathbb{R}^2: g(y_1,y_2)<a_1\}.
\end{align*}
Again, it is clear that the set $\mathcal{S}$ is bounded, convex and symmetric about the origin, and hence it contains a nontrivial lattice point in $\mathbb{Z}^2$ if
\begin{align}
    \vol(\mathcal{S})>2^2\text{disc}(\mathbb{Z}^2)=4. \label{ineq8}
\end{align}
The volume of $\mathcal{S}$ is given by
\begin{align*}
    \vol(\mathcal{S})=\frac{\pi}{\sqrt{\frac{d}{4}\left(1-\frac{a_2^2d}{a_1^2}\right)}},
\end{align*}
and so if $a \in \mathcal{F}_K$, then
\begin{align*}
    \sqrt{\frac{d}{4}\left(1-\frac{a_2^2d}{a_1^2}\right)} \geq \frac{\sqrt{d}}{2}\sqrt{1-\tanh(R_K)^2}=\frac{\sqrt{\Delta_K}}{2\cosh(R_K)},
\end{align*}
and so if inequality \eqref{ineq8} holds,
\begin{align*}
    \frac{\pi}{4} > \frac{\sqrt{\Delta_K}}{2\cosh(R_K)},
\end{align*}
as required.
\\
We have therefore verified that if inequality \eqref{ineq4} holds, then the field $K=\mathbb{Q}(\sqrt{d})$ is not unit reducible, for any square-free positive integer $d$. Let $u=v_1+v_2\sqrt{d}>1$ denote the fundamental unit of $K$. Then
\begin{align*}
    2\cosh(R_K)=e^{R_K}+e^{-R_K}=v_1+v_2\sqrt{d}+\frac{1}{v_1+v_2\sqrt{d}}=\begin{cases}
    &2v_1, \hspace{2mm} \text{if} \hspace{2mm} \norm_{K/\mathbb{Q}}(u)=1,
    \\
    &2v_2\sqrt{d} \hspace{2mm} \text{if} \hspace{2mm} \norm_{K/\mathbb{Q}}(u)=-1
    \end{cases},
\end{align*}
so that inequality \eqref{ineq4} can be rewritten as
\begin{align}
    \frac{\pi}{4}>\begin{cases}
    \frac{\sqrt{d}}{2v_1}, \hspace{2mm} &\text{if} \hspace{2mm} d \equiv 1 \mod 4, \hspace{1mm} \norm_{K/\mathbb{Q}}(u)=1;
    \\ \frac{\sqrt{d}}{2v_2\sqrt{d}}=\frac{1}{2v_2}, \hspace{2mm} &\text{if} \hspace{2mm} d \equiv 1 \mod 4, \hspace{1mm} \norm_{K/\mathbb{Q}}(u)=-1;
    \\ \frac{2\sqrt{d}}{2v_1}=\frac{\sqrt{d}}{v_1}, \hspace{2mm} &\text{if} \hspace{2mm} d \equiv 2,3 \mod 4, \hspace{1mm} \norm_{K/\mathbb{Q}}(u)=1;
    \\ \frac{2\sqrt{d}}{2v_2\sqrt{d}}=\frac{1}{v_2}, \hspace{2mm} &\text{if} \hspace{2mm} d \equiv 2,3 \mod 4, \hspace{1mm} \norm_{K/\mathbb{Q}}(u)=-1.
    \end{cases}\label{ineq9}
\end{align}
Suppose first that $d \equiv 2,3 \mod 4$. If $\norm_{K/\mathbb{Q}}(u)=v_1^2-dv_2^2=-1$ and $v_2=1$, then $d=v_1^2+1$ for some integer $v_1$. Clearly then $d \equiv 2 \mod 4$, but this would imply that $d$ is of the type $T_1$, for which we have already shown $\mathbb{Q}(\sqrt{d})$ is unit reducible, hence we must have $v_2>1$. Then
\begin{align*}
    \frac{1}{v_2} \leq \frac{1}{2}<\frac{\pi}{4},
\end{align*}
which verifies inequality \eqref{ineq9}, and so $K$ must not be unit reducible. Now, assume that $\norm_{K/\mathbb{Q}}(u)=v_1^2-dv_2^2=1$. Similarly, if $v_2=1$ then $d=v_1^2-1$, and so $d \equiv 3 \mod 4$ which implies that $d$ is of type $T_2$, for which we have already shown that $\mathbb{Q}(\sqrt{d})$ is unit reducible, hence we must have $v_2>1$. Then
\begin{align*}
    v_1^2=1+v_2^2d > 4d,
\end{align*}
and so
\begin{align*}
    \frac{\sqrt{d}}{u_1} <\frac{\sqrt{d}}{2\sqrt{d}}=\frac{1}{2}<\frac{\pi}{4},
\end{align*}
which verifies inequality \eqref{ineq9}, and so $K$ must not be unit reducible, which covers all cases for $d \equiv 2,3 \mod 4$.
\\
Now suppose that $d \equiv 1 \mod 4$. Note that we must have $v_1,v_2 \in \frac{1}{2}\mathbb{Z}$, so we may express $v_1=\frac{V_1}{2}, v_2=\frac{V_2}{2}$ for some $V_1,V_2 \in \mathbb{Z}$. If we assume that $\norm_{K/\mathbb{Q}}(u)=-1$, then
\begin{align*}
    V_1^2-dV_2^2=-4.
\end{align*}
If $V_2=1$, then $d=V_1^2+4$ which would imply that $d$ is of type $T_3$, for which we have already shown $\mathbb{Q}(\sqrt{d})$ is unit reducible, and so we assume that $V_2>1$. Then
\begin{align*}
    \frac{1}{2v_2}=\frac{1}{V_2} \leq \frac{1}{2}<\frac{\pi}{4},
\end{align*}
which verifies inequality \eqref{ineq9}, and so $K$ must not be unit reducible. Now, assume that $\norm_{K/\mathbb{Q}}(u)=1$ so that
\begin{align*}
    V_1^2-dV_2^2=4.
\end{align*}
If $V_2=1$, then $d=V_1^2-4$ which would imply that $d$ is of type $T_4$, for which we have already shown $\mathbb{Q}(\sqrt{d})$ is unit reducible, and so we assume that $V_2>1$. Then
\begin{align*}
    V_1^2=4+dV_2^2>4d,
\end{align*}
and so
\begin{align*}
    \frac{\sqrt{d}}{2v_1}=\frac{\sqrt{d}}{V_1}<\frac{\sqrt{d}}{2\sqrt{d}}=\frac{1}{2}<\frac{\pi}{4},
\end{align*}
which verifies inequality \eqref{ineq9}, and so $K$ must not be unit reducible. This covers every case, which verifies the statement of the theorem.
\end{proof}
\subsection{Simplest cubic fields}
Let $K$ be a cubic number field with defining polynomial $P(x)$. $K$ is said to be a \emph{simplest cubic field} if $P(x)$ has the form
\begin{align*}
    P(x)=P_t(x)=x^3-tx^2-(t+3)x-1,
\end{align*}
where $t$ is some integer, and we denote by $K_t=\mathbb{Q}(\theta_t)$. where $\theta_t$ satisfies $P_t(\theta_t)=0$. Without loss of generality, we may assume that $t \geq -1$, since $K_t=K_{-(t+3)}$. The following facts are well-known:
\begin{itemize}
    \item $K_t/\mathbb{Q}$ is a cyclic cubic extension.
    \item $\theta_t$ is a unit, satisfying $\sigma_1(\theta_t)=\theta_t,\sigma_2(\theta_t)=-\frac{1+\theta_t}{\theta_t}$, $\sigma_3(\theta_t)=\sigma_2^2(\theta_t)=\frac{-1}{1+\theta_t}$, where $\sigma_i \in \text{Gal}(K_t/\mathbb{Q})$, which is a cyclic Galois group.
\end{itemize}
Denote by $\sqrt{\Delta_t} \triangleq t^2+3t+9$ and $\mathcal{O}_{K_t}$ the ring of integers of $K_t$. It is known that $\mathcal{O}_{K_t}=\mathbb{Z}[\theta_t]$ if and only if $\Delta_t$ is equal to the discriminant of $K_t$ \cite{kas20}. From now on, we will denote by $K=K_t$, $\theta=\theta_t$ where the context is clear, unless we need to clarify the value of $t$ for brevity. We will also only be focusing on the case where $\mathcal{O}_K=\mathbb{Z}[\theta_t]$. We will now prove the following theorem.
\begin{theorem}
Suppose that $K$ is a simplest cubic field, and that $\mathcal{O}_K=\mathbb{Z}[\theta]$. Then $K$ is unit reducible.
\end{theorem}
\begin{proof}
First, we will prove the following useful lemma.
\begin{lemma}\label{lm4}
Let $a$ denote an arbitrary totally positive element in $K$. Then $a \in \mathcal{F}_K$ if and only if
\begin{align}
    \trace_{K/\mathbb{Q}}(a) \leq \trace_{K/\mathbb{Q}}(a\sigma_2^i(\theta)^2), \label{ineq10}
\end{align}
for $i \in \{0,1,2\}$.
\end{lemma}
\begin{proof}
Since $\pm 1, \theta, \sigma_2(\theta)$ generate the unit group, and unit may be rewritten as $u=\pm \theta^k \sigma_2(\theta^l)$ for some $k,l \in \mathbb{Z}$. First, let's assume that $l=0$. If $k=2$, then
\begin{align*}
    &\trace_{K/\mathbb{Q}}(u^2a)=\trace_{K/\mathbb{Q}}(\theta^4 a)
    \\&=(t^2+t+3)\trace_{K/\mathbb{Q}}(\theta^2 a)+(t^2+3t+1)\trace_{K/\mathbb{Q}}(\theta a)+t\trace_{K/\mathbb{Q}}(a)
    \\&=\frac{1}{2}(t^2+3t+1)\trace_{K/\mathbb{Q}}((1+\theta)^2a)+\frac{1}{2}(t^2-t+5)\trace_{K/\mathbb{Q}}(\theta^2 a)-\frac{1}{2}(t^2+t+1)\trace_{K/\mathbb{Q}}(a)
    \\& =\frac{1}{2}(t^2+3t+1)\trace_{K/\mathbb{Q}}(\theta^2 \sigma_2(\theta)^2 a)+\frac{1}{2}(t^2-t+5)\trace_{K/\mathbb{Q}}(\theta^2 a)-\frac{1}{2}(t^2+t+1)\trace_{K/\mathbb{Q}}(a)
    \\& \geq \trace_{K/\mathbb{Q}}(\theta^2 a),
\end{align*}
given the assumption that $t$ is positive, and assuming inequality \eqref{ineq10} holds.
Now, assume that we have $\trace_{K/\mathbb{Q}}(\theta^{2(k-1)}a) \leq \trace_{K/\mathbb{Q}}(\theta^{2k}a)$ for some $k \geq 2$. Then
\begin{align*}
    &\trace_{K/\mathbb{Q}}(\theta^{2k+2}a)=\trace_{K/\mathbb{Q}}(\theta^{2k-1}(t\theta^2+(t+3)\theta+1)a)
    \\&=t\trace_{K/\mathbb{Q}}(\theta^{2k+1}a)+(t+3)\trace_{K/\mathbb{Q}}(\theta^{2k}a)+\trace_{K/\mathbb{Q}}(\theta^{2k-1}a)
    \\&=\frac{1}{2}t\trace_{K/\mathbb{Q}}(\theta^{2k}(1+\theta)^2a)-\frac{1}{2}t\trace_{K/\mathbb{Q}}(\theta^{2(k+1)} a)+\frac{1}{2}(t+5)\trace_{K/\mathbb{Q}}(\theta^{2k} a)\\&-\frac{1}{2}\trace_{K/\mathbb{Q}}(\theta^{2(k-1)}a)+\frac{1}{2}\trace_{K/\mathbb{Q}}(\theta^{2(k-1)}(1+\theta)^2a)
    \\ & \geq \frac{1}{2}t\trace_{K/\mathbb{Q}}(\theta^{2k}(1+\theta)^2a)-\frac{1}{2}t\trace_{K/\mathbb{Q}}(\theta^{2(k+1)} a)+\frac{1}{2}(t+4)\trace_{K/\mathbb{Q}}(\theta^{2k} a)\\&+\frac{1}{2}\trace_{K/\mathbb{Q}}(\theta^{2(k-1)}(1+\theta)^2a),
\end{align*}
and so
\begin{align*}
    &\frac{1}{2}(t+2)\trace_{K/\mathbb{Q}}(\theta^{2k+2}a)\\& \geq \frac{1}{2}t\trace_{K/\mathbb{Q}}(\theta^{2k}(1+\theta)^2a)+\frac{1}{2}(t+4)\trace_{K/\mathbb{Q}}(\theta^{2k} a)+\frac{1}{2}\trace_{K/\mathbb{Q}}(\theta^{2(k-1)}(1+\theta)^2a)
    \\& \geq \frac{1}{2}(t+4)\trace_{K/\mathbb{Q}}(\theta^{2k} a),
\end{align*}
hence $\trace_{K/\mathbb{Q}}(\theta^{2k+2}a) \geq \trace_{K/\mathbb{Q}}(\theta^{2k} a)$ for all $k \geq 0$. Now, suppose that $u=\theta^k \sigma_2(\theta^l)$ for nonzero $k,l$. We may assume without loss of generality that $k \neq l$, as otherwise $\theta^{2k}\sigma_2(\theta)^{2l}=\sigma_3(\theta^{-1})^{2k}$, and treating this case is identical to before. First, note that
\begin{align*}
    &\trace_{K/\mathbb{Q}}(\theta^4\sigma_2(\theta)^2 a)=\trace_{K/\mathbb{Q}}((\theta^2+2\theta^3+\theta^4)a)
    \\&=(t^2+3t+4)\trace_{K/\mathbb{Q}}(\theta^2 a)+(t^2+4t+4)\trace_{K/\mathbb{Q}}(\theta a)+ (t+2)\trace_{K/\mathbb{Q}}(a)
    \\& =\frac{1}{2}(t^2+4t+4)\trace_{K/\mathbb{Q}}((1+\theta)^2a)+\frac{1}{2}(t^2+2t+4)\trace_{K/\mathbb{Q}}(\theta^2 a)-\frac{1}{2}(t^2+2t)\trace_{K/\mathbb{Q}}(a)
    \\& \geq \frac{1}{2}(t^2+4t+4)\trace_{K/\mathbb{Q}}((1+\theta)^2a)=\frac{1}{2}(t^2+4t+4)\trace_{K/\mathbb{Q}}(\theta^2 \sigma_2(\theta)^2 a) \geq 2\trace_{K/\mathbb{Q}}(\theta^2 \sigma_2(\theta)^2 a),
\end{align*}
given the assumption that $t$ is positive and $\trace_{K/\mathbb{Q}}(\theta^2 a) \geq \trace_{K/\mathbb{Q}}(a)$. Now, assume that $k>l \geq 1$, and that
\begin{align*}
    \trace_{K/\mathbb{Q}}(\theta^{2k}\sigma_2(\theta)^{2l} a) \geq \trace_{K/\mathbb{Q}}(\theta^{2(k-1)}\sigma_2(\theta)^{2l}a).
\end{align*}
Then
\begin{align*}
    &\trace_{K/\mathbb{Q}}(\theta^{2(k+1)}\sigma_2(\theta)^{2l}a)
    \\&=\dots = \frac{1}{2}t\trace_{K/\mathbb{Q}}(\theta^{2k}(1+\theta)^2\sigma_2(\theta)^{2l}a)-\frac{1}{2}t\trace_{K/\mathbb{Q}}(\theta^{2(k+1)}\sigma_2(\theta)^{2l} a)+\frac{1}{2}(t+5)\trace_{K/\mathbb{Q}}(\theta^{2k}\sigma_2(\theta)^{2l} a)\\&-\frac{1}{2}\trace_{K/\mathbb{Q}}(\theta^{2(k-1)}\sigma_2(\theta)^{2l}a)+\frac{1}{2}\trace_{K/\mathbb{Q}}(\theta^{2(k-1)}(1+\theta)^2\sigma_2(\theta)^{2l}a)
    \\& \geq \frac{1}{2}t\trace_{K/\mathbb{Q}}(\theta^{2k}(1+\theta)^2\sigma_2(\theta)^{2l}a)-\frac{1}{2}t\trace_{K/\mathbb{Q}}(\theta^{2(k+1)}\sigma_2(\theta)^{2l} a)+\frac{1}{2}(t+4)\trace_{K/\mathbb{Q}}(\theta^{2k}\sigma_2(\theta)^{2l} a)\\&+\frac{1}{2}\trace_{K/\mathbb{Q}}(\theta^{2(k-1)}(1+\theta)^2\sigma_2(\theta)^{2l}a),
\end{align*}
and so
\begin{align*}
    &\frac{1}{2}(t+2)\trace_{K/\mathbb{Q}}(\theta^{2k+2}\sigma_2(\theta)^{2l}a)\\& \geq \frac{1}{2}t\trace_{K/\mathbb{Q}}(\theta^{2k}(1+\theta)^2\sigma_2(\theta)^{2l}a)+\frac{1}{2}(t+4)\trace_{K/\mathbb{Q}}(\theta^{2k} \sigma_2(\theta)^{2l}a)+\frac{1}{2}\trace_{K/\mathbb{Q}}(\theta^{2(k-1)}(1+\theta)^2\sigma_2(\theta)^{2l}a)
    \\& \geq \frac{1}{2}(t+4)\trace_{K/\mathbb{Q}}(\theta^{2k}\sigma_2(\theta)^{2l} a),
\end{align*}
so $\trace_{K/\mathbb{Q}}(\theta^{2k+2}\sigma_2(\theta)^{2l}a) \geq \trace_{K/\mathbb{Q}}(\theta^{2k}\sigma_2(\theta)^{2l} a)$. We are now ready to prove that our system of inequalities is sufficient to show that $\trace_{K/\mathbb{Q}}(u^2a) \geq \trace_{K/\mathbb{Q}}(a)$ for all $u \in \mathcal{O}_K^\times$. Suppose that $u=\theta^k\sigma_2(\theta)^l$ for some integers $k,l$. It clearly holds that our method may be replicated for $\sigma_2(u)$ and $u^{-1}$ by a switch of variables. Therefore, by using the table below, we show that our system of inequalities \eqref{ineq10} fully defines the reduction domain.
\begin{center}
\centering
\begin{tabular}{||c c||}
\hline
Values of $(k,l)$                 & Inequality      \\[0.5ex] 
\hline
\hline 
$k$ positive, $l=0$               & $\trace_{K/\mathbb{Q}}(\theta^{2k}\theta^{2l}a) \geq \trace_{K/\mathbb{Q}}(\theta^{2(k-1)} \sigma_2(\theta)^{2l}a)$                                                                                               \\ 
$k$ positive, $l$ positive, $k=l$ & $\trace_{K/\mathbb{Q}}(\theta^{2k} \sigma_2(\theta)^{2l}a) \geq \trace_{K/\mathbb{Q}}(\theta^{2(k-1)} \sigma_2(\theta)^{2(l-1)}a$                                                                                          \\ 
$k$ positive, $l$ positive, $k>l$ & $\trace_{K/\mathbb{Q}}(\theta^{2k} \sigma_2(\theta)^{2l}a) \geq \trace_{K/\mathbb{Q}}(\theta^{2(k-1)} \sigma_2(\theta)^{2l}a)$                                                                                            \\ 
$k$ positive, $l$ positive, $k<l$ & \makecell{$\trace_{K/\mathbb{Q}}(\theta^{2k} \sigma_2(\theta)^{2l}a)=\trace_{K/\mathbb{Q}}(\theta^{-2(l-k)}\sigma_3(\theta)^{-2l}a)$  \\ $\geq \trace_{K/\mathbb{Q}}(\theta^{-2(l-k)} \sigma_2(\theta)^{-2(l-1)}a)$}                          \\ 
$k$ positive, $l$ negative        & \makecell{$\trace_{K/\mathbb{Q}}(\theta^{2k} \sigma_2(\theta)^{2l}a)=\trace_{K/\mathbb{Q}}( \sigma_2(\theta)^{-2(k+|l|)}\sigma_3(\theta)^{-2k}a)$ \\ $\geq \trace_{K/\mathbb{Q}}(\sigma_2(\theta)^{-2(k+|l|-1)} \sigma_3(\theta)^{-2k}a)$}     \\ 
\hline
\end{tabular}
\end{center}
\end{proof}
We have therefore shown that the system of inequalities \eqref{ineq10} is sufficient to fully define the reduction domain $\mathcal{F}_K$ of any simplest cubic field with ring of integers $\mathbb{Z}[\theta]$. Let $a=a_0+a_1\theta+a_2\sigma_2(\theta) \in K_{>>0}$ denote an arbitrary totally positive element of $K$. Then if $a \in \mathcal{F}_K$, it must hold that
\begin{align*}
    &(t^2+2t+3)a_0+(t^3+3t^2+8t+3)a_1-(t^2+4t+6)a_2 \geq 0,
    \\&(t^2+2t+3)a_0+(3-t)a_1+(t^3+3t^2+8t+3)a_2 \geq 0,
    \\&(t^2+2t+3)a_0-(t^2+4t+6)a_1+(3-t)a_2 \geq 0,
    \\&(t^2+4t+6)a_0+(t^3+5t^2+13t+15)a_1-(t^2+5t+12)a_2 \geq 0,
    \\&(t^2+4t+6)a_0-(2t+3)a_1+(t^3+5t^2+13t+15)a_2 \geq 0,
    \\&(t^2+4t+6)a_0-(t^2+5t+12)a_1-(2t+3)a_2 \geq 0.
\end{align*}
The rays representing this region are given by
\begin{align*}
    &r_1=3+\theta+(t+2)\sigma_2(\theta),
    \\&r_2=t^2+2t+3-(t+2)\theta-(t+1)\sigma_2(\theta),
    \\&r_3=t+3+(t+1)\theta-\sigma_2(\theta),
    \\&s_1=t^2+t+6 -t\theta+3\sigma_2(\theta),
    \\&s_2=t^2+4t+6-3\theta-(t+3)\sigma_2(\theta),
    \\&s_3=t+6+(t+3)\theta+t\sigma_2(\theta).
\end{align*}
It is clear that the sets $R=\{r_1,r_2,r_3\}$, $S=\{s_1,s_2,s_3\}$ are invariant under the action of the Galois group. For any $x=x_0+x_1\theta+x_2\sigma(\theta) \in \mathbb{Z}[\theta]$, we have
\begin{align*}
    \trace_{K/\mathbb{Q}}(r_1 x^2)&=(t^2+3t+9)\left(x_0^2+x_1^2+(t^2+2t+3)x_2^2+2(t+1)x_0x_2-2x_1x_2\right),
\end{align*}
for which the nonzero minimum value of the quadratic form above is clearly $t^2+3t+9$, which occurs at the points $(1,0,0)$, $(0,1,0)$ and $(-t,1,1)$, which all correspond to units.
\\
Likewise, we have
\begin{align*}
    \trace_{K/\mathbb{Q}}(s_1x^2)&=(t^2+3t+9)(2x_0^2+2x_1^2+(t^2+3t+5)x_2^2-2x_0x_1+2(t+2)x_0x_2-2x_1x_2)
    \\&=(t^2+3t+9)((x_0-x_1+x_2)^2+(x_0+(t+1)x_2)^2+x_1^2+(t+3)x_2^2),
\end{align*}
for which we can deduce that the nonzero minimum of the quadratic form is $2(t^2+3t+9)$, which occurs at the points $(1,0,0),(0,1,0)$, and the third successive minimum value is $(t^2+3t+9)(t^2+t+5)$, which occurs at the point $(-t,1,1)$. All of these vectors correspond to units in the ring of integers.
\\
In what follows is that $\trace_{K/\mathbb{Q}}(r_1 x^2)$ and 
$\trace_{K/\mathbb{Q}}\left(\frac{s_1}{2} x^2\right)$ have the same non-zero minimum. Obviously the unary forms $ r_1x^2$ and $\frac{s_1}{2}x^2$ are not equivalent, because 
\[
\frac{s_1}{2} = \frac{-3\theta^2 + 2t\theta + t^2 + 4t + 12}{2} \not\in \mathbb{Z}[\theta]
\] but $r_1\in \mathbb{Z}[\theta] $ and $r_1u^2 \in \mathbb{Z}[\theta]$ for all $u \in \mathbb{Z}[\theta]^\times$.

Now, since the groups $R,S$ are invariant under the action of the Galois group, it is immediately clear that the minima of the forms induced by $r_2,r_3,s_2,s_3$ also occur at units. Moreover, since every element in $\mathcal{F}_K$ can be represented by
\begin{align*}
    \lambda_1r_1+\lambda_2r_2+\lambda_3r_3+\lambda_4s_4+\lambda_5s_5+\lambda_6s_6,
\end{align*}
for $\lambda_i \geq 0$, $\sum_{i=1}^6 \lambda_i>0$, every element in $\mathcal{F}_K$ must attain their minima at unit values. Since every element in $K_{>>0}$ is equivalent to an element in $\mathcal{F}_K$, the proof of the theorem follows.
\end{proof}
\section{Unit Reducibility in General Totally Real Number Fields}
Throughout this section, for any totally real number field $K$, we will denote by $K_\mathbb{R}=K \otimes \mathbb{R}$. It is easy to see that $K_\mathbb{R} \cong \mathbb{R}^n$, if $[K:\mathbb{Q}]=n$. Though we have so far focused on specific families of number fields, we may also ascertain some results for arbitrary number fields.
\begin{definition}
A \emph{lattice} $\Lambda$ in $\mathbb{R}^n$ is a discrete subgroup of $\mathbb{R}^n$. If
\begin{align*}
    \Lambda=\bigoplus_{i=1}^m \mathbb{Z}\mathbf{b}_i,
\end{align*}
for some $\mathbf{b}_i \in \mathbb{R}^n$, $m \leq n$. We say that $\Lambda$ is full-rank if $m=n$, and we will assume that all lattices are full-rank from now on. We will denote by $\det(\Lambda)$ the absolute value of the determinant of the matrix composed of the basis vectors of $\Lambda$. Denote by
\begin{align*}
    J_p \triangleq \{\mathbf{x} \in \mathbb{R}^n: \|\mathbf{x}\|_p \leq 1\},
\end{align*}
where $\|\cdot\|_p$ denotes the $l_p$ norm of an element in $\mathbb{R}^n$. Then we define the $i$th successive minima with respect to the $l_p$-norm of a lattice $\Lambda$ by
\begin{align*}
    \lambda_i^{(p)}(\Lambda) \triangleq \argmin_{\lambda>0}\{\lambda J_p \hspace{1mm} \text{contains $i$ linearly independent lattice vectors.}\}
\end{align*}
When considering the $l_2$ norm, we will simply write $\lambda_i(\Lambda)$ instead of $\lambda_i^{(2)}(\Lambda)$.
\end{definition}

From Hadamard's inequality one can show the following lemma:
\begin{lemma}\label{lm2}
For all full-rank lattices $\Lambda$ in $\mathbb{R}^n$,
\begin{align*}
    \lambda_n(\Lambda) \geq \det(\Lambda)^{1/n}.
\end{align*}
\end{lemma}
\begin{proof}
Let $\mathbf{b}_1,\mathbf{b}_2,\dots,\mathbf{b}_n$ denote the basis for $\Lambda$, and denote by $\mathbf{b}_i(i)$ the vector $\mathbf{b}_i$ after being orthogonalised with respect to the space generated by $\mathbf{b}_1,\dots,\mathbf{b}_{i-1}$, for all $1<i \leq n$. Then
\begin{align*}
    \det(\Lambda)=\prod_{i=1}^n\|\mathbf{b}_i(i)\|,
\end{align*}
where $\mathbf{b}_1(1)=\mathbf{b}_1$. Assume that
\begin{align*}
    \mathbf{v}_i=\sum_{j=1}^n x_j^{(i)}\mathbf{b}_j,
\end{align*}
for some $x_j^{(i)} \in \mathbb{Z}$, where $\|\mathbf{v}_i\|=\lambda_i(\Lambda)$, for all $1 \leq i \leq n$. We may assume without loss of generality that the $\mathbf{b}_i$ are arranged so that the largest $j$ such that $x_j^{(i)}$ is nonzero is greater than or equal to $i$, otherwise we may readjust the order until this holds, since $\mathbf{v}_1,\mathbf{v}_2,\dots,\mathbf{v}_n$ must be linearly independent (e.g. we cannot have three $\mathbf{v}_i$ expressible as the nonzero integer sum of only two $\mathbf{b}_i$, or else they would not be linearly independent, which is a contradiction). Now, suppose that for some $i$ we have
\begin{align*}
    \mathbf{v}_i=\sum_{j=1}^kx_j^{(i)}\mathbf{b}_j,
\end{align*}
where $i<k\leq n$ and $x_k^{(i)} \neq 0$. Let $\gamma_i=\gcd(x_i^{(i)},x_{i+1}^{(i)},\dots,x_k^{(i)})$. Then it is possible to construct a new basis containing $\mathbf{b}_1,\dots,\mathbf{b}_{i-1},\frac{1}{\gamma_i}\sum_{j=1}^kx_j^{(i)}\mathbf{b}_j$. This argument can be recursively applied until we have a basis $\mathbf{b}_1,\dots,\mathbf{b}_n$ satisfying
\begin{align*}
    \mathbf{v}_i=\sum_{j=1}^ix_j^{(i)}\mathbf{b}_j,
\end{align*}
where $x_i^{(i)}$ is nonzero for all $1 \leq i \leq n$. In this way, it is easy to see that
\begin{align*}
    \|\mathbf{v}_i\|^2=\lambda_i(\Lambda)^2\geq {x_i^{(i)}}^2\|\mathbf{b}_i(i)\|^2\geq \|\mathbf{b}_i(i)\|^2,
\end{align*}
and so letting $\max_{i}\|\mathbf{b}_i(i)\|=\|\mathbf{b}_l(l)\|$, we have
\begin{align*}
    \lambda_n(\Lambda)^n \geq \lambda_l(\Lambda)^n \geq \|\mathbf{b}_l(l)\|^n\geq \prod_{i=1}^n \|\mathbf{b}_i(i)\|=\det(\Lambda).
\end{align*}
\end{proof}
\begin{lemma}\label{lm3}
Let $U_K$ denote the set of units such that $\trace_{K/\mathbb{Q}}(a)\leq \trace_{K/\mathbb{Q}}(u^2a) $ defines a non-redundant boundary of $\mathcal{F}_K$ for all $u \in \mathcal{F}_K$. Then the multiplicative group using the elements of $U_K$ as generators generates the entire unit group of $K$. 
\end{lemma}
\begin{proof}
It is easily seen that $\min_{x \in \mathcal{O}_K \setminus \{0\}}\trace_{K/\mathbb{Q}}(x^2)=\trace_{K/\mathbb{Q}}(1)=[K:\mathbb{Q}]$, for any totally real number field $K$, and that $\min_{x \in \mathcal{O}_K \setminus \{0,\pm 1\}}\trace_{K/\mathbb{Q}}(x^2)>[K:\mathbb{Q}]$. Then, if some unit $v$ cannot be generated by the elements of $U_K$ under multiplication, we must have $\trace_{K/\mathbb{Q}}(v^2u^2)>\trace_{K/\mathbb{Q}}(1)$ for all $u$ in the multiplicative group generated by $U_K$, which is a contradiction, since then $\trace_{K/\mathbb{Q}}(v^2u^2)$ is not reduced for any $u$. 
\end{proof}
\begin{lemma}\label{conj1}
Define $U_K$ as before. Then for all $u \in U_K$, there exists an $a$ in $\mathcal{F}_K$ such that $\trace_{K/\mathbb{Q}}(a)=\trace_{K/\mathbb{Q}}(u^2a)$.
\end{lemma}
\begin{proof}
Obviously $1$ satisfies $\trace_{K/\mathbb{Q}}(1) \leq \trace_{K/\mathbb{Q}}(u^2)$ for all $u\in U_K$. Since $\mathcal{F}_K$ is a convex cone with finitely many facets, $U_K$ is a finite set. Obviously there exists 
a nonzero integer $k$ such that $\trace_{K/\mathbb{Q}}(u^{2k}) > \trace_{K/\mathbb{Q}}(u^{2k}\cdot u^2)$. Let 
\[
a= u^{2k}\trace_{K/\mathbb{Q}}(1-u^2) + \trace_{K/\mathbb{Q}}(u^{2k}(u^2-1)).
\]
Then 
\begin{align*}
 \trace_{K/\mathbb{Q}}(a) &=    n\trace_{K/\mathbb{Q}} (u^{2k+2}) - \trace_{K/\mathbb{Q}} (u^{2})\trace_{K/\mathbb{Q}} (u^{2k}), \\
  \trace_{K/\mathbb{Q}}(au^2) &=    n\trace_{K/\mathbb{Q}} (u^{2k+2}) - \trace_{K/\mathbb{Q}} (u^{2})\trace_{K/\mathbb{Q}} (u^{2k})
\end{align*}
as desired.
\end{proof}

Finally, we are able to prove the following theorem.
\begin{theorem}
Suppose that $K$ is a totally real number field of degree $n$ over $\mathbb{Q}$. Denote respectively by $R_K$ and $\Delta_K$ the regulator and discriminant of $K$. Then $K$ is not unit reducible if 
\begin{align*}
    \frac{|\Delta_K|}{1+\frac{1}{2}e^{\frac{2}{\sqrt{n}}R_K^{\frac{1}{n-1}}}} \leq \frac{\left(\frac{n\pi}{4}\right)^n}{\Gamma\left(\frac{n}{2}+1\right)^2},
\end{align*}
where $\Gamma$ denotes the gamma function.
\end{theorem}
\begin{proof}
Suppose that $a$ is a totally positive element of $K$. We may assume without loss of generality that $a \in \mathcal{F}_K$. Then, consider the sphere of radius $\sqrt{\trace_{K/\mathbb{Q}}(a)}$ and dimension $n$, given by $\mathcal{B}_n\trace_{K/\mathbb{Q}}(a)^{\frac{n}{2}}$, where $\mathcal{B}_n$ denotes the ball of unit radius. By Minkowski's theorem, since the volume of the lattice generated by $\trace_{K/\mathbb{Q}}(ax^2)$ is $\sqrt{|\Delta_K|\norm_{K/\mathbb{Q}}(a)}$, $K$ is not unit reducible if we can find an $a \in \mathcal{F}_K$ satisfying
\begin{align*}
    \vol(\mathcal{B}_n\trace_{K/\mathbb{Q}}(a)^{\frac{n}{2}})=\frac{\pi^{\frac{n}{2}}}{\Gamma\left(\frac{n}{2}+1\right)}\trace_{K/\mathbb{Q}}(a)^{\frac{n}{2}}>2^n\sqrt{|\Delta_K|\norm_{K/\mathbb{Q}}(a)},
\end{align*}
and so we need
\begin{align}\label{ineq12}
    \trace_{K/\mathbb{Q}}(a)>\frac{4}{\pi}|\Delta_K|^{\frac{1}{n}}\norm_{K/\mathbb{Q}}(a)^{\frac{1}{n}}\Gamma\left(\frac{n}{2}+1\right)^{\frac{2}{n}},
\end{align}
for some $a \in \mathcal{F}_K$. Now, if Lemma \ref{conj1} holds, by Lemma \ref{lm3}, there must exist at least $n-1$ nontrivial units $u_1,u_2,\dots,u_{n-1}$ such that $\{u_1,\dots,u_{n-1}\}$ constitute a multiplicative basis for the unit group, and that there exists $a_1,a_2,\dots,a_{n-1} \in \mathcal{F}_K$ such that
\begin{align*}
    \trace_{K/\mathbb{Q}}(a_i)=\trace_{K/\mathbb{Q}}(u_i^2a_i),
\end{align*}
for each $1 \leq i \leq n-1$. Then, for each $i$, using the AGM inequality,
\begin{align*}
    2\trace_{K/\mathbb{Q}}(a_i)=\trace_{K/\mathbb{Q}}((1+u_i^2)a_i)\geq n\norm_{K/\mathbb{Q}}((1+u_i^2))^{\frac{1}{n}}\norm_{K/\mathbb{Q}}(a_i)^{\frac{1}{n}}.
\end{align*}
Note that
\begin{align*}
    &\norm_{K/\mathbb{Q}}((1+u_i^2))=\prod_j \sigma_j(1+u_i^2)=2+\trace_{K/\mathbb{Q}}(u_i^2)+(\text{other positive terms}) \\&\geq 2+\trace_{K/\mathbb{Q}}(u_i^2) > 2+\max_{j} \sigma_j(u_i^2).
\end{align*}
Denote by $\mathcal{L}=\{(\log(|\sigma_1(v)|),\log(|\sigma_2(v)|),\dots,\log(|\sigma_n(v)|)), \forall v \in \mathcal{O}_K^\times\}$ the \emph{log-unit lattice} (clearly this satisfies the axioms of a lattice by our previous definition), whose covolume is $\vol(\mathcal{L})=\sqrt{n}R_K$. Then for some $1 \leq i \leq n$, it must hold that \begin{align*}&\|(\log(|\sigma_1(u_i)|),\log(|\sigma_2(u_i)|),\dots,\log(|\sigma_n(u_i)|))\|_{\infty} \\&\geq \frac{1}{\sqrt{n}}\|(\log(|\sigma_1(u_i)|),\log(|\sigma_2(u_i)|),\dots,\log(|\sigma_n(u_i)|))\| \geq \frac{1}{\sqrt{n}}\lambda_{n-1}(\mathcal{L}),\end{align*} so
\begin{align*}
    &2\trace_{K/\mathbb{Q}}(a_i) > n(2+\max_{\sigma} \sigma(u_i^2))^{\frac{1}{n}}\norm_{K/\mathbb{Q}}(a_i)^{\frac{1}{n}}=n\left(2+e^{2\log(\max_{\sigma} |\sigma(u_i)|)}\right)^{\frac{1}{n}}\norm_{K/\mathbb{Q}}(a_i)^{\frac{1}{n}}
    \\&\geq n\left(2+e^{\frac{2}{\sqrt{n}}\lambda_{n-1}(\mathcal{L})}\right)^{\frac{1}{n}}\norm_{K/\mathbb{Q}}(a_i)^{\frac{1}{n}}\geq n\left(2+e^{\frac{2}{\sqrt{n}}R_K^{\frac{1}{n-1}}}\right)^{\frac{1}{n}}\norm_{K/\mathbb{Q}}(a_i)^{\frac{1}{n}},
\end{align*}
which yields the required result.
\end{proof}

\section{Perfect Unary Forms}
We will now demonstrate how the study of unit reducible fields has direct application to the study of perfect unary forms.
\\
\begin{proof}[Proof of Theorem \ref{thm1}.]
Let $K=\mathbb{Q}(\sqrt{d})$ denote a real quadratic field. As in section \ref{sec2}, we will use the notation $\sigma_2$ to denote the field automorphism that sends $x_1+x_2\sqrt{d}$ to $x_1-x_2\sqrt{d}$, for any $x_1+x_2\sqrt{d} \in K$. For any $a \in K_{>>0}$,
\begin{align}
    av^2+\sigma_2(av^2)>0, \label{ineq11}
\end{align}
for all $v \in \mathcal{O}_K \setminus \{0\}$.\\
Suppose that $K$ is unit reducible. Consider the quadratic form $x^2$. This form is clearly not perfect, since $\mathcal{M}(1)=\{ \pm 1\}$ and inequality \eqref{ineq11} has rank $1$. Hence, there exists a perfect unary form $ax^2$ such that $\mu(a)=\mu(1)=2$ and $\{\pm 1\} \subset \mathcal{M}(a)$. Let $v \in \mathcal{M}(a) \setminus \{ \pm 1\}$. Since $K$ is unit reducible, we may assume that $v$ is a unit. Therefore the system of linear equations
\begin{align}\label{ineq12}
    &z \cdot 1^2+ w\cdot 1^2=t,
    \\& z \cdot v^2+ w \cdot \sigma_2(v^2)=t
\end{align}
has a unique solution $(z,w)$. Using Cramer's rule, we have
\begin{align*}
    z=\frac{\begin{vmatrix}t & 1
    \\ t & \sigma_2(v^2)
    \end{vmatrix}}{\begin{vmatrix}1 & 1 \\ v^2 & \sigma_2(v^2)\end{vmatrix}} \hspace{2mm} \text{and} \hspace{2mm} w=\frac{\begin{vmatrix}1 & t
    \\ v^2 & t
    \end{vmatrix}}{\begin{vmatrix}1 & 1 \\ v^2 & \sigma_2(v^2)\end{vmatrix}}.
\end{align*}
Without loss of generality, one can assume that $t \in \mathbb{Q}$, so 
\begin{align*}
    \sigma_2\left(\begin{vmatrix}
    1 & t
    \\
    v^2 & t
    \end{vmatrix}\right)
    =-\begin{vmatrix}
    t & 1
    \\
    t & \sigma_2(v^2)
    \end{vmatrix},\hspace{2mm}
    \sigma\left(\begin{vmatrix}
    1 & 1
    \\
    v^2 & \sigma_2(v^2)
    \end{vmatrix}\right)=-\begin{vmatrix}
    1 & 1
    \\
    \sigma_2(v^2) & v^2
    \end{vmatrix},
\end{align*}
hence $w=\sigma_2(z)$ (i.e. $w$ is a field conjugate of $z$).
\\
Now, if $\norm_{K/\mathbb{Q}}(v)=1$, an immediate calculation shows that $z=rv^{-1}$ for some positive rational $r$. If $u$ is the fundamental generator of the unit group of $K$ modulo roots of unity, then we must have $v=u^k$ for some integer $k$. Given that the form $rv^{-1}$ is necessarily reduced, it holds that $v \in \{u,u^{-1}\}$.
\\
If $\norm_{K/\mathbb{Q}}(v)=-1$, an immediate calculation shows that $z=\sqrt{d}sv^{-1}$ for some positive rational $s$. Similarly to before, we deduce that $v \in \{u,u^{-1}\}$.
\\
Now, given that the system of linear equations \eqref{ineq12} is invariant under the replacement of $v$ by $\sigma_2(v)$, if the unary form $v^{-1}x^2$ (resp. $\sqrt{d}v^{-1}x^2$) is perfect, then so is $\sigma_2(v^{-1})x^2$ (resp. $\sigma_2(\sqrt{d}v^{-1})x^2$. However, $v^{-1}x^2=v(v^{-1}x)^2$, so the forms $vx^2, v^{-1}x^2$ are equivalent (resp. $v^{-1}\sqrt{d}x^2, v\sqrt{d}x^2$ are equivalent forms).
\\
Recall that two unary forms $ax^2, bx^2$ are called \emph{neighbouring forms} if $a \neq b$, $\mu(a)=\mu(b)$ and $\mathcal{M}(a) \cap \mathcal{M}(b) \neq \emptyset$. It is known that the graph of neighbouring forms is connected \cite{mar03}. First, we consider the neighbouring forms of $u^{-1}x^2$, when $u$ is totally positive and $u>1$. We have shown that $\mathcal{M}(u^{-1})=\{\pm 1, \pm u\}$. Since $\sigma_2(u^{-1})=u$, $\mathcal{M}(u)=\{\pm 1, \pm u^{-1}\}$. But the intersection of $\mathcal{M}(u)$ and $\mathcal{M}(u^{-1})$ is non-empty (equal to $\{\pm 1\}$), so $ux^2,u^{-1}x^2$ are neighbouring forms. Moreover, we have already shown they are also equivalent. Next, we will multiply both forms by $u^{-2}$ to obtain the forms $u^{-1}x^2$ and $u^{-3}x^2$. Obviously, $\mathcal{M}(u^{-3})\cap \mathcal{M}(u^{-1})=\{\pm 1\}$, so the forms are neighbours, and moreover they are equivalent. Thus, all neighbouring forms of the perfect form $u^{-1}x^2$ are equivalent to $u^{-1}x^2$. From this, we conclude that there exists only one class of unary perfect forms up to equivalence and scaling.
\\
To complete the proof, we need to consider the case if $K$ is not unit reducible. Then, there exist $a \in K_{>>0}$ and $v \in \mathcal{O}_K$ such that
\begin{align*}
    \trace_{K/\mathbb{Q}}(av^2)<\trace_{K/\mathbb{Q}}(a),
\end{align*}
and $v$ is not a unit. It follows from the work of Koecher \cite{koe60} that there exists a $b \in K_{>>0}$ such that $\mu(a)=\mu(b)$ and $bx^2$ is a perfect unary form. Obviously, $1 \not\in \mathcal{M}(b)$. Thus, there exists a perfect form $b^\prime x^2$ such that $\mu(b)=\mu(b^\prime)$. But $b^\prime x^2$ is not equivalent to $bx^2$, so there must be at least two classes of unary perfect forms, which completes the proof.
\end{proof}
\begin{proof}[Proof of Theorem \ref{thm2}.]
Throughout, we will assume that $K$ is not unit reducible, as we have already treated this case. First, we consider the case $d \equiv 2,3 \mod 4$. Consider the set $\{\pm 1, \pm (n-\sqrt{d})\}$. The system of linear equations
\begin{align*}
    z+w&=1,
    \\
    z(n-\sqrt{d})^2+w(n+\sqrt{d})^2&=1
\end{align*}
has a unique solution
\begin{align*}
    z=\frac{(n+\sqrt{d})^2-1}{4n\sqrt{d}}, \hspace{2mm} w=\sigma_2(z).
\end{align*}
Let
\begin{align*}
    a=2nd+(n^2+d-1)\sqrt{d}.
\end{align*}
It is totally positive since
\begin{align*}
    \norm_{K/\mathbb{Q}}(a)=4n^2d^2-(n^2+d-1)^2d=4n^4+(6r-r^2-1)n^2+(2r^2-r^3-1)>0,
\end{align*}
if $n \geq 2$. 
\\
We will now consider the minimum of $ax^2$. Let
\begin{align*}
    f(x_1,x_2)=\trace_{K/\mathbb{Q}}(a(x_1+x_2\sqrt{d})^2)=4d(nx_1^2+(n^2+d-1)x_1x_2+ndx_2^2)\triangleq 4dg(x_1,x_2).
\end{align*}
Now
\begin{align*}
    g(x_1-nx_2,x_2)&=n(x_1-nx_2)^2+(2n^2+r-1)(x_1-nx_2)x_2+n(n^2+r)x_2^2
    \\&=nx_1^2+(r-1)x_1x_2+nx_2^2 \triangleq h(x_1,x_2).
\end{align*}
The real quadratic field $K$ is of R--D type, so $-n \leq r-1 < n$ and thus the binary quadratic form $h(x_1,x_2)$ is reduced, so the minimum of $f$ is $4nd$. Moreover, the minimal vectors of $g$ (and of $f$) are $\pm(1,0)$ and $\pm(n,-1)$. From this, we have that $ax^2$ attains its (trace) minimum at $\pm 1$ and $\pm (n-\sqrt{d})$. Hence, $ax^2$ is a perfect unary form.
\\
Now, consider when $d \equiv 1 \mod 4$ and $n$ odd. Consider the set $\left\{\pm 1, \pm \frac{n-\sqrt{d}}{2}\right\}$. As before, the system of linear equations
\begin{align*}
    z+w&=1,
    \\
    z\left(\frac{n-\sqrt{d}}{2}\right)^2+w\left(\frac{n+\sqrt{d}}{2}\right)^2&=1
\end{align*}
has a unique solution
\begin{align*}
    z=\frac{\frac{1}{4}(n+\sqrt{d})^2-1}{n\sqrt{d}}, \hspace{2mm} w=\sigma_2(z).
\end{align*}
Let
\begin{align*}
    a=2nd+(n^2+d-4)\sqrt{d}.
\end{align*}
$a$ is totally positive, since
\begin{align*}
    &\norm_{K/\mathbb{Q}}(a)\\&=d(4n^4+4n^2r+16n^2+8r-4n^4-4n^2r-r^2-16)=d(16n^2+8r-r^2-16)>0,
\end{align*}
if $n \geq 2$. 
\\
Consider the minimum of $ax^2$. Let
\begin{align*}
    f(x_1,x_2)&=\trace_{K/\mathbb{Q}}\left(a\left(x_1+x_2\frac{1+\sqrt{d}}{2}\right)^2\right)\\&=d\left(4nx_1^2+2(2n+n^2+d-4)x_1x_2+(n+nd+n^2+d-4)x_2^2\right)\triangleq dg(x_1,x_2).
\end{align*}
Now, letting $p=\frac{n+1}{2}$,
\begin{align*}
    &g(x_1-px_2,x_2)=4n(x_1-px_2)^2+(2n^2+n+d(n+2)-8)x_2^2+2(n^2+2n+d-4)(x_1-px_2)x_2
    \\&=4nx_1^2+2(n^2+2n+d-4-4np)x_1x_2\\&+\left(2n^2+n+d(n+2)-8+4np^2-2p(n^2+2n+d-4))\right)x_2^2
    \\&=4nx_1^2+2(r-4)x_1x_2+(2n^2+(4-r)n+(5r-8))x_2^2,
\end{align*}
and again it is easy to see that this form is reduced for all $n \geq 2$. Given that $\frac{n-\sqrt{d}}{2}$ is a rational scalar multiplied by $n-\sqrt{d}$, all further arguments throughout the proof can be easily extended to this case too, and so we omit treating the case $d \equiv 1 \mod 4$ throughout the rest of the proof.
\\
Now, if $ax^2$ is perfect, then $\sigma_2(a)x^2$ is perfect too. Moreover, the minimal vectors of $\sigma_2(a)x^2$ are $\mathcal{M}(\sigma_2(a))=\{\pm 1, \pm (n+\sqrt{d})\}$.
\\
Again, if $bx^2$ is a perfect neighbour of $ax^2$, then $\sigma_2(b)x^2$ is a perfect neighbour of $\sigma_2(a)x^2$ by the symmetry of the action of the Galois group. Since $1$ is invariant under the Galois action, $ax^2$, $\sigma_2(a)x^2$ share the common minimum $\pm 1$. Thus, $ax^2$, $\sigma_2(a)x^2$ are perfect neighbours. To find another neighbour, we consider the action of the unit group on $\{\pm(n-\sqrt{d})\}$. It is known that if $|r| \neq 1,4$ then the fundamental unit is (see \cite{yuk68})
\begin{align*}
    u_0=\frac{n^2+d+2n\sqrt{d}}{r}.
\end{align*}
Immediate calculations yield that
\begin{align*}
    (n-\sqrt{d})u_0=n+\sqrt{d}=\sigma_2(n-\sqrt{d}).
\end{align*}
However, $(\sigma_2(a)u_0^2)x^2$ attains its minimum at $\pm u_0^{-1}$ and $\pm u_0^{-1}(n+\sqrt{d})=(n-\sqrt{d})$, so the corresponding perfect neighbour is equivalent to $\sigma_2(a)x^2$. From this, we conclude that all perfect neighbours of $ax^2$ are equivalent to $\sigma_2(a)x^2$.
\\
Since $K$ is not unit reducible, the perfect unary forms $ax^2, \sigma_2(a)x^2$ are not equivalent.
\end{proof}
\begin{proof}[Proof of Theorem \ref{thm3}]
See Appendix.
\end{proof}

\section*{Appendix}
We will continue to use notations as in section 2.2. First note that $r_2=r_1\sigma_2(\theta_t)^{-2}$ and 
$r_3=r_1 \theta_t^2$, hence $r_1\sim r_2 \sim r_3$. Also $s_3=s_1\theta_t^2\sigma_2(\theta_t)^2$ and $s_2=s_1\theta_t^2$.

\begin{proof}[Proof of Theorem \ref{thm3}.]
The simplest cubic field is a vector space over $\mathbb{Q}$ equipped with a non-degenerate bilinear form $a.b=\trace_{K/\mathbb{Q}}(ab)$.

Since there are only two equivalence classes of perfect ray forms, so we apply the Vorono\"i algorithm for enumerating perfect neighbouring form to $r_1$ and $s_1$ (because equivalent perfect forms have equivalent perfect neighbours). 

\noindent\textbf{Perfect neighbours of $r_1$.} Recall that the trace minimum of
\[
r_1 = \left((-t - 2)\theta_t^2 + (t^2 + 2t + 1)\theta_t + (t^2 + 4t + 7)\right)x^2=
(3+\theta_t + (t+2)\sigma_2(\theta_t))x^2.
\]
is attained at $v_1=-(t+1)+\theta_t+\sigma_2(\theta_t)$, $v_2=1$ and $v_3=\theta_t$. Denote by $\Pi_{r_1}$ the 
convex cone generated by $v_1^2$, $v_2^2$ and $v_3^2$. The hyperplane 
generated by $v_i^2$ and $v_j^2$ will be denoted by $L(v_i,v_j)$.
\begin{itemize}
    \item \underline{Perfect neighbour along the facet $L(v_2, v_3) \cap \Pi_{r_1}$.} After solving the system of linear equations
    \begin{align}
        \trace_{K/\mathbb{Q}}(\psi_1 (v_2^2))&=0, \label{eq:r1ls1a} \\
        \trace_{K/\mathbb{Q}}(\psi_1 (v_3^2))&=0 \label{eq:r1ls1b}
    \end{align}
    subject to $\trace_{K/\mathbb{Q}}(\psi_1 (v_3^2))>0$ we obtain 
    \begin{align*}
    \psi_1(x)&=\big((-t^3 - 5t^2 - 12t - 11)\theta_t^2 - (t^4 + 5t^3+ 13t^2 + 14t + 4)\theta_t \\& - (t^4 + 7t^3 + 24t^2 + 42t + 31)\big)x^2 \\
    &= \left((2t^2 + 7t + 9)+(t^2 + 3t + 4)\theta_t +(t^3 + 5t^2 + 12t + 11)\sigma_2(\theta_t)\right)x^2.
    \end{align*}
    With respect to the integral basis $1,\theta_t, \sigma_2(\theta_t)$
    the facet vector $\psi_1(x)=(p_0+p_1\theta_t+p_2\sigma_2(\theta_t))x^2$ 
    is a solution to the following system of linear inequalities
    \[
    \begin{pmatrix}
    t^2+ 2t+ 6 & t^3+ 3t^2+ 9t+ 3 & -t^2- 3t- 6 \cr
  t^2+ 4t+ 9 & -t^2- 4t- 12 & -t- 3
    \end{pmatrix}
    \begin{pmatrix}
    p_0 \\ p_1 \\ p_2
    \end{pmatrix}=0.
    \]
    Since 
    \[
    \trace_{K/\mathbb{Q}}(\psi_1(v_3)) = t^4 + 6t^3 + 21t^2 + 36t + 27 >0 ,
    \]
    for all $t\geq -1$, so $\psi_1$ is a facet vector.
    But
    \[
    N_1=r_1+\frac{1}{2} \psi_1 = \frac{1}{2} s_1 \sigma_2(\theta_t)^2 \sim \frac{1}{2} s_1 
    \]
    hence the perfect neighbour $N_1$ is homothetic to already known perfect form $s_1$.
    \item \underline{Perfect neighbour along the facet $L(v_1, v_3) \cap \Pi_{r_1}$.}
    The facet vector $\psi_2$ must satisfy the following conditions
    \begin{align}
        \trace_{K/\mathbb{Q}}(\psi_2 (v_1^2))&=0,\label{eq:r1ls2a}\\
        \trace_{K/\mathbb{Q}}(\psi_2 (v_3^2))&=0, \label{eq:r1ls2b}\\
        \trace_{K/\mathbb{Q}}(\psi_2 (v_2^2))&>0\label{eq:r1ls2c}
    \end{align}
    (see \cite[Proposition 7.1.10]{mar03}). From 
    \begin{equation}
    \label{eq:r1:matEq}
    \begin{pmatrix}
    3&t&t\cr
  t^2+ 4\*t+ 9 & -t^2- 4t- 12 & -t- 3
    \end{pmatrix}
    \begin{pmatrix}
    p_0 \\ p_1 \\ p_2
    \end{pmatrix}=0,
    \end{equation}
    where  $\psi_2(x)=(p_0+p_1\theta_t+p_2\sigma_2(\theta_t))x^2$, we obtain
    \begin{align*}
    \psi_2(x) &= \left((t + 4)\theta_t^2 + (-t^2 - 3t + 1)\theta_t + (-t^2 - 5t - 8)\right)x^2\\&= \left(t+(t+1)\theta -(t+4)\sigma_2(\theta)\right)x^2.
    \end{align*}
    (The matrix equation (\ref{eq:r1:matEq}) was a matrix form of 
    the equations (\ref{eq:r1ls2a}) and (\ref{eq:r1ls2b})). It is straightforward to check 
    \[
    \trace_{K/\mathbb{Q}}(\psi_2(v_2)) = t^4 + 6t^3 + 21t^2 + 36t + 27>0 \quad \mbox{for all $t\geq -1$.}
    \]
    Direct calculations show that
    \[
    N_2=r_1+\frac{1}{2} \psi_2 = \frac{1}{2} s_3.
    \]
    \item \underline{Perfect neighbour along the facet $L(v_1, v_2)\cap \Pi_{r_1}$.} The facet vector $\psi_3$ is a solution to the system of linear equations
    \begin{align}
        \trace_{K/\mathbb{Q}}(\psi_3 (v_1^2))&=0,\label{eq:r1ls3a}\\
        \trace_{K/\mathbb{Q}}(\psi_3 (v_2^2))&=0 \label{eq:r1ls3b}
    \end{align}
    subject to $\trace_{K/\mathbb{Q}}(\psi_3 (v_3^2))>0$. Let 
    $\psi_3(x)=(p_0+p_1\theta_t+p_2\sigma_2(\theta_t))x^2$. Then
    \[
    \begin{pmatrix}
    3&t&t\\
  t^2+ 2t+ 6 & t^3+ 3t^2+ 9t+ 3 & -t^2- 3t- 6
    \end{pmatrix}
    \begin{pmatrix}
    p_0 \\ p_1 \\ p_2
    \end{pmatrix}=0
    \]
    and
    \begin{align*}
    \psi_3(x) &= \left((2t + 1)\theta_t^2 - (2t^2 + 2t + 2)\theta_t - (t^2 + 4t + 2)\right)x^2 \\&= \left(t^2+t - (t+2)\theta_t -(2t+1)\sigma_2(\theta_t)\right)x^2. 
    \end{align*}
    But $\psi_3$ is a facet vector, because 
    \[
    \trace_{K/\mathbb{Q}}(\psi_3(v_3))= t^4 + 6t^3 + 21t^2 + 36t + 27>0.
    \]
    The corresponding perfect neighbour of $r_1$ is
    \[
    N_3 = r_1+\frac{1}{2} \psi_3 = \frac{1}{2} s_1.
    \]
\end{itemize}
\noindent\textbf{Perfect neighbours of $s_1$.} Recall that the trace minimum of
\[
s_1(x)= (-3\theta_t^2 + 2t\theta_t + (t^2 + 4t + 12))x^2=
(t^2+t+6 -t\theta_t+3\sigma_2(\theta_t))x^2
\]
is attained at $v_1,v_2$ and $v_4=1+\theta_t$. 
\begin{itemize}
    \item \underline{Perfect neighbour along the facet $L(v_1, v_4) \cap \Pi_{s_1}$.} Again, the conditions for the facet vector $\psi_4(x)$ are (see \cite[Proposition 7.1.10]{mar03})
    \begin{align}
        \trace_{K/\mathbb{Q}}(\psi_2 (v_1^2))&=0,\label{eq:s1ls1a}\\
        \trace_{K/\mathbb{Q}}(\psi_2 (v_3^2))&=0, \label{eq:s1ls1b}\\
        \trace_{K/\mathbb{Q}}(\psi_2 (v_2^2))&>0\label{eq:s1ls1c}.
    \end{align}
    With respect to the indeterminants $p_0$, $p_1$ and $p_3$, where
    $psi_4(x)=(p_0+p_1\theta_t + p_2\sigma_2(\theta_t))x^2$, we get the matrix equation
    \[
    \begin{pmatrix}
    3&t&t\cr
  t^2+ 4t+ 9 & t^3+ 5t^2+ 14t+ 15 & -t^2- 4t- 12
    \end{pmatrix}
    \begin{pmatrix}
        p_0 \\ p_1 \\ p_2
    \end{pmatrix}=0.
    \]
    From this we obtain 
    \begin{align*}
    \psi_4(x)&= \left( (2t + 5)\theta_t^2 - (2t^2 + 6t + 4)\theta_t - (t^2 + 12t + 10) \right) x^2\\&= 
    \left((t^2+3t)-(t+4)\theta_t -(2t+5)\sigma_2(\theta_t)\right)x^2.
    \end{align*}
    Since 
    \[
    \trace_{K/\mathbb{Q}}(\psi_4(v_2))=2(t^2+3t+9)>0,
    \]
    we have that $\psi_4$ is a facet vector. Moreover,
    \[
    N_4=s_1+\psi_4 = 2r_1\sigma_2(\theta_t)^{-2} \sim 2r_1.
    \]
    \underline{Perfect neighbour along the facet $L(v_2, v_4)\cap \Pi_{s_1}$.}
    The facet vector $\psi_5$ must satisfy
    \begin{align}
        \trace_{K/\mathbb{Q}}(\psi_5 (v_2^2))&=0,\label{eq:s1ls2a}\\
        \trace_{K/\mathbb{Q}}(\psi_5 (v_4^2))&=0, \label{eq:s1ls2b}\\
        \trace_{K/\mathbb{Q}}(\psi_5 (v_1^2))&>0\label{eq:s1ls2c}.
    \end{align}
    Let $\psi_5(x)=\left(p_0+p_1\theta_t + p_2\sigma_2(\theta_t) \right)x^2$. After solving  
    \[
    \begin{pmatrix}
     \left(t^2+ 2t + 6\right) & \left(t^3+ 3t^2+ 9t+ 3\right) & \left(-t^2- 3t- 6\right) \cr
  \left(t^2+ 4t+ 9\right) & \left(t^3 + 5t^2+ 14t+ 15\right) & \left(-t^2- 4t- 12\right)
    \end{pmatrix}
     \begin{pmatrix}
        p_0 \\ p_1 \\ p_2
    \end{pmatrix}=0
    \]
    with respect to $p_0$, $p_1$, $p_2$, we obtain
    \[
    \psi_5(x) = \left(-7\theta_t^2 + (6t + 2)\theta_t + (t^2 + 6t + 20)\right) x^2 = \left((t^2-t+6)+(-t+2)\theta_t +7\sigma_2(\theta_t) \right)x^2.
    \]
    The unary form $\psi_5(x)$ is a facet vector because
    \[
    \trace_{K/\mathbb{Q}}(\psi_5 (v_1^2)) = 2(t^2+3t+9)>0.
    \]
    We finally have
    \[
    N_5 = s_1+\psi_5=\theta_t^{-2} \sigma_2(\theta_t)^{-2} 2r_1 \sim 2r_1.
    \]
    \underline{Perfect neighbour along the facet $L(v_1, v_2)\cap \Pi_{s_1}$.} Since $\Pi_{s_1}\cap L(v_1,v_2) = \Pi_{r_1}\cap L(v_1,v_2)$ so $2r_1$ and $s_1$ are perfect neighbours along the facet  $L(v_1, v_2)\cap \Pi_{s_1}$. 
\end{itemize}
Since all perfect neighbours $N_1, N_2,\ldots,N_5$ are equivalent to 
(up to homothety) perfect forms $r_1$ or $s_1$. 
\end{proof}

\begin{corollary}
Let $K=\mathbb{Q}(\theta_t)$ be a simplest cubic field. If $ax^2$ is 
a positive definite unary quadratic over $K$, then 
$M(a)\subset \mathbb{Z}[\theta_t]^\times$.
\end{corollary}
\end{document}